\documentclass[12pt]{amsart}%\documentclass[12pt,a4paper]{amsart}
\usepackage{amsfonts}
\usepackage{hyperref}

\usepackage[letterpaper]{geometry}

\DeclareMathOperator{\Inf}{Inf}
\DeclareMathOperator{\supp}{supp}
\usepackage{amssymb}
\newtheorem{thm}{\bf Theorem}[section]
\newtheorem{lemma}[thm]{\bf Lemma}

\theoremstyle{definition}

\begin{document}

\title[Estimation of function's supports under arithmetic constraints]{Estimation of function's supports under arithmetic constraints}

\author[Norbert Hegyv\'ari]{Norbert Hegyv\'ari}
 \address{Norbert Hegyv\'{a}ri, ELTE TTK,
E\"otv\"os University, Institute of Mathematics, H-1117
P\'{a}zm\'{a}ny st. 1/c, Budapest, Hungary and Alfr\'ed R\'enyi Institute of Mathematics, Hungarian Academy of Science, H-1364 Budapest, P.O.Box 127.}
\email{hegyvari@renyi.hu}

\begin{abstract}
The well-known $|supp(f)||supp(\widehat{f}|\geq |G|$ inequality gives lower estimation of each supports. In the present paper we give upper estimation under arithmetic constrains.
The main notion will be the additive energy which  plays a central role in additive combinatorics. We prove an uncertainty inequality that shows a trade-off between the total changes of the indicator function of a subset $A\subseteq \mathbb F^n_2$ and the additive energy of $A$ and the Fourier spectrum.

\bigskip


\end{abstract}

 \maketitle

\section{Introduction}

A classical uncertainty principle says that
\begin{equation}\label{0}
|supp(f)||supp(\widehat{f}|\geq |G|,
\end{equation}
where $G$ is a finite abelian group, $f$ acts on $G$ and $\widehat{f}$ is the Fourier transform of $f$. Roughly speaking if $|supp(f)|$ is "small', then $|supp(\widehat{f}|$ should be "big" and vice versa. In fact (\ref{0}) gives a quantitative {\it lower bound} for the other support too. But if one of the supports is "medium" size, it does not follow that the other support is "medium" size.

We can ask whether is there any other relation between the two supports? For example, is there a slowly increasing function $F$ for which $|supp(f)|\leq F(|supp(\widehat{f}|)$?  Certainly (\ref{0}) shows that we should not expect such a relationship in general. 
However, one might expect that if we assume some {\it structure} on $supp(f)$, we might expect to get such a relation. The aim of this note is to provide such an estimation. The two links that help us come from additive combinatorics and theoretical computer science.  (see section 2 and 3).

In the present paper $G$ will be $\mathbb F^n_2$,  where $\mathbb{F}_2$ is the finite field with $2$ elements. Let $A\subseteq \mathbb F^n_2$. Let $f$ be the indicator of $A$, i.e. let
$$
f(x) = \left\{
     \begin{array}{lr}
       1 &  x \in A\\
       0 &  x \notin A
     \end{array}
\right.,
$$
so $A=supp(f)$ is the support of $f$ and this type of function is said to be Boolean, since it maps from $\{0,1\}^n$ to $\{0,1\}$.

There are many other interesting uncertainty inequalities; we mention just two: the Heisenberg uncertainty principle for conjugate variables in quantum mechanics and a useful inequality in prime fields; if $f: \mathbb{Z}_p\mapsto \mathbb{Z}_p$, then  $|\supp f| + |\supp \widehat{f}|\geq p$, where $p$ is the characteristic of the prime field.

In the last decades there are several interplay between complexity theory and additive combinatorics. One of the most interesting example is connection between notions in computer sciences and the {\it Gowers norm} (see e.g. \cite{3,6}).
Another interesting example is an additive communication complexity problem which is supported by an example of Behrend on the maximal density of a set not containing three-term arithmetic progression (see e.g. \cite{2}).

\bigskip

\section{Preliminaries, Notations}

\medskip

The first concept we will discuss comes from additive combinatorics. Let $\mathbb F^n_2$ be the $n-$dimensional vector space with characteristic $2$.  For a set $A\subseteq \mathbb F^n_2$, the {\it additive energy} of $A$ is defined as the number of quadruples  $(a_1,a_2,a_3,a_4)$ for which $a_1+a_2=a_3+a_4$, formally $E(A):=|\{(a_1,a_2,a_3,a_4)\in A^4: \ a_1+a_2=a_3+a_4\}|$.
Clearly $|A|^2\ll E(A)\ll |A|^3$ holds, since the quadruple $(a_1,a_2,a_1,a_2)$ is always a solution and given $a_1,a_2,a_3$ the term $a_4$ is uniquely determined by them. (We will use the notation $|X|\ll |Y|$ to denote the estimate $|X|\leq C|Y|$ for some absolute constant $C>0$). This concept was introduced by Terence Tao, and plays a central role in additive combinatorics. (See, e.g., \cite{4,5}.)

Let $f: \{0,1\}^n \to \{0,1\}$ and $g: \{0,1\}^n \to \{0,1\}$  be two functions. The expected value of $f$ is
$$
\mathbb{E}(f):=\frac{1}{2^n}\sum_{x\in \{0,1\}^n}f(x),
$$
and the inner product of $f$ and $g$ is $\langle f,g\rangle:=\mathbb{E}(fg)$. For $S\subseteq [n]$ the corresponding  vector is $x=(x_1,x_2,\dots,x_n)\in \{0,1\}^n$ namely $x_i=1$ if $i\in S$ and $x_i=0$ otherwise. A basis function or character is defined by $\chi_x(y):=(-1)^{\langle x,y\rangle_2}$, where $\langle x,y\rangle_2:=\sum_{i=1}^nx_iy_i \pmod 2$.

For functions $f$ and $g$ their convolution is defined by $f\ast g (x):=\mathbb{E}f(y)g(x+y)$.
It is easy to verify that the convolution is associative: $f\ast (g\ast h)=(f\ast g)\ast h$.

For a set $S\subseteq [n]$ ($[n]:=\{1,2,\dots,n\}$), the Fourier transform of $f$ is $\widehat{f}(S)=\langle f,\chi_S\rangle$.

For the Fourier transform the following are true:
$$
(i)\  \langle f,g\rangle =\sum_{r\in \{0,1\}^n}\widehat{f}(r)\widehat{g}(r) \ \text{(Plancherel)}
$$
$$
(ii) \ \|f\|_2^2=\mathbb{E}(f^2)=\langle f,f\rangle = \sum_{r\in \{0,1\}^n}\widehat{f}^2(r) \ \text{(Parseval)}
$$
So by the Parseval formula for the indicator function we have 
\begin{equation}\label{Par}
\sum_{r\in \{0,1\}^n}\widehat{A}^2(r)=\frac{1}{2^n}|A|
\end{equation}.

The second link comes from computer science. The influence $Inf_i(f)$ of the $i^{th}$ variable on a Boolean function $f$ is the probability $\Inf_i(f)=\Pr_{x\in \{0,1\}^n}[f(x)\neq f(x+e_i)]$, where $x$ is uniformly distributed over $\{0,1\}^n$, and $f(x+e_i)$ means that we change the $i^{th}$ coordinate of $x$ to the opposite, (i.e. 0 to 1 or 1 to 0). The total influence of $f$ is defined to be $I(f):=\sum_i \Inf_i(f)$.

We write $supp(\widehat{f}(S))=\{S\subseteq [n]: \widehat{f}(S)\neq 0\}$.

\section{An uncertainty inequality}
\label{sec:uncert}

As in the introduction we mentioned one might expect that $|supp(f)|\leq |supp(\widehat{f}|^B$ for some $B>0$ if we assume some {\it structure} on $supp(f)$. In the next theorem, we will show that the additive energy is the hidden quantity that provides this connection between the two supports.

We prove 

\begin{thm}\label{c1} Let $A:=supp(f)$, where $f: \mathbb F^n_2 \mapsto \{0,1\}$. Assume that $E(A)=|A|^{2+\eta}/n$ ($0<\eta<1$). We have 
$$
|supp(f)|\ll|supp(\widehat{f})|^{\frac{2}{1-\eta}}
$$
if $supp(f)$ is large enough (e.g. $n/|supp(f)|\to 0$ as $n\to \infty $).
\end{thm}
We will derive this theorem form our main result:
\begin{thm}\label{T3}
Let $f:\{0,1\}^n \to \{0,1\}$ and $A\subseteq \{0,1\}^n$ be its support, i.e. let $A=f^{-1}(1)$. Denote by $E(A)$ the additive energy of $A$. Then
\begin{equation}\label{1}
\frac{3|A|^3}{128E(A)}<I(f)|supp(\widehat{f})|^2.
\end{equation}
\end{thm}
Theorem \ref{c1} follows now from $I(f)\leq n$.

\begin{proof}[Proof of Theorem \ref{T3}]

We consider the truncated sum of $\sum_S\widehat{f}^2(S)$ as follows: let $0<R$ be any real number chosen be later. By (\ref{Par}) we get
$$
\frac{1}{2^n}|A|=\sum_{S\subseteq [n]}\widehat{f}^2(S)\leq  \sum^{\lfloor R\rfloor}_{S\subseteq [n]; |S|=1}(1- |S|/R)\widehat{f}^2(S)+\frac{1}{R}\sum_{S\subseteq [n]}|S|\widehat{f}^2(S).
$$
Now we need the following (see e.g. \cite{1})
\begin{lemma}
For every Boolean function we have
$I(f)=4\sum_{S\subseteq [n]}|S|\widehat{f}^2(S)$
\end{lemma}
\begin{proof}
Recall that $Inf_i(f)= \mathbb{E}_x[(f(x)- f((x+e_i))^2]$. 
Take now the Fourier representations of $f(x)$ and  $f(x+e_i)$ we have $|f(x)- f(x+e_i)|=2\sum_{i\in S}\widehat{f}(S)(-1)^{\langle S,x\rangle}$. Indeed, we have two cases; if $i\not \in S$, a term in $f(x)$ cancels the term in $f(x+e_i)$. If $i \in S$ it will be the double. Expanding the expectation we have
$$
Inf_i(f)= \mathbb{E}_x[f(x)- f((x+e_i))^2]=\langle (f(x)- f(x+e_i),f(x)- f(x+e_i)\rangle=
$$
$$
=4\langle \sum_{i\in S}\widehat{f}(S)\chi_S,\sum_{i\in T}\widehat{f}(T)\chi_T\rangle=4\sum_{i\in S}|\widehat{f}(S)|^2.
$$
Finally
$$
I(f)=\sum_iI_i(f)=4\sum_i\sum_{i\in S}|\widehat{f}(S)|^2=
$$
$$
=4\sum_S\sum_{i}|\widehat{f}(S)|^2=4\sum_{S}|S||\widehat{f}(S)|^2.
$$
\end{proof}

So we have
$$
\frac{1}{2^n}|A|\leq \sum_{S\in supp(\widehat{f})}\sum^{\lfloor R\rfloor}_{|S|=1}(1- |S|/R)\widehat{f}^2(S)+\frac{1}{4R} I(f)\leq |supp(\widehat{f})|\sum^{\lfloor R\rfloor}_{|S|=1}(1- |S|/R)\widehat{f}^2(S)+\frac{1}{4R} I(f),
$$
 Now let us use the Cauchy inequality for $\sum^{\lfloor R\rfloor}_{|S|=1}(1- |S|/R)\widehat{f}^2(S)$. We obtain
$$
\sum^{\lfloor R\rfloor}_{|S|=1}(1- |S|/R)\widehat{f}^2(S)\leq \sqrt{\sum^{\lfloor R\rfloor}_{|S|=1}(1- |S|/R)^2}\sqrt{\sum_{S\subseteq [n]}\widehat{f}^4(S)}.
$$
Now we state a well-known lemma (see e.g. in \cite{1}):
\begin{lemma}
\begin{equation}\label{2}
\frac{1}{2^{3n}}E(A)=\mathbb{E}_{x,y,z}(f(x)f(y)f(z)f(x+y+z))=\sum_r\widehat{f}^4(r).
\end{equation}
\end{lemma}
For sake of completeness we give the short proof.

\begin{proof}
First note that in $\mathbb{F}_2$, $x+(x+y+z)=y+z$ holds, thus we have
$$
\mathbb{E}_{x,y,z}(f(x)f(y)f(z)f(x+y+z))=
$$
$$
=\frac{1}{2^{3n}}|\{(a_1,a_2,a_3,a_4)\in A^4: \ a_1+a_2=a_3+a_4\}|=\frac{1}{2^{3n}}E(A).
$$
Furthermore using that $\mathbb{E}_z(f(z)f(x+y+z))=f\ast f(x+y)$, we have
$$
\mathbb{E}_{x,y,z}(f(x)f(y)f(z)f(x+y+z))=\mathbb{E}_x(f(x)\mathbb{E}_y(f(y)\mathbb{E}_z(f(z)f(x+y+z))))=
$$
$$
=\mathbb{E}_x(f(x)\mathbb{E}_y(f(y)f\ast f(x+y)))=\mathbb{E}_x(f(x)(f\ast(f\ast f(x)))).
$$
Write briefly $f_3\ast (r)$ instead of $(f\ast (f\ast f))(r)$. By the Plancherel formula, the associative of the convolution, and the Fourier transformation of a convolution we have
$$
\mathbb{E}_{x,y,z}(f(x)f(y)f(z)f(x+y+z))=\sum_r[\widehat{f\cdot f_3\ast (r)}]=
$$
$$
=\sum_r\widehat{f}(r)\widehat{f_3\ast (r)}=\sum_r\widehat{f}(r)\widehat{f}^3(r)=\sum_r\widehat{f}^4(r).
$$
Let $A:=f^{-1}(1)$. Then $\sum_r\widehat{f}^4(r)=\frac{1}{2^{3n}}E(A)$. As we detected $\sum_r\widehat{f}^4(r)$ can be written as $\mathbb{E}_{x,y,z}(f(x)f(y)f(z)f(x+y+z))$ and we are done.

\end{proof}

Expanding the sum $\sum^{\lfloor R\rfloor}_{|S|=1}(1-|S|/R)^2$ and by the elementary summation we have
$$
\sum^{\lfloor R\rfloor}_{|S|=1}(1-|S|/R)^2\leq \frac{R}{3}
$$
Summing up
\begin{equation}\label{3}
\frac{1}{2^n}|A|=\sum_{S\subseteq [n]}\widehat{f}^2(S)\leq |supp(\widehat{f})|\sqrt{\frac{R}{3}}\sqrt{\frac{1}{2^{3n}}E(A)}+\frac{1}{4R} I(f).
\end{equation}
The first term is an increasing, the second term is an decreasing function of $R$. So when the two terms are the same, i.e. when $|supp(\widehat{f})|\sqrt{\frac{R}{3}}\sqrt{\frac{1}{2^{3n}}E(A)}=\frac{1}{4R} I(f)$ (and where the bound is optimal), we get for $R$:
$$
R=\sqrt[3]{\frac{3\cdot{2^{3n}}I^2(f)}{16E(A)|supp(\widehat{f})|^2}}.
$$
Thus by (\ref{3}) we have
$$
\frac{1}{2^n}|A|\leq 2\frac{1}{R}I(f)=2\sqrt[3]{\frac{16E(A)|supp(\widehat{f})|^2}{3\cdot{2^{3n}}I^2(f)}}I(f).
$$
Taking the cube we obtain our result.
\end{proof}

\medskip

\noindent{\bf Acknowledgement.} This work was supported by the National Research, Development and Innovation Fund of Hungary through project no. grant K-129335.

\medskip

\

\end{document}